\documentclass[a4paper]{article}
\usepackage{amsmath,amsthm,amssymb,enumerate,hyperref,bm,color,eurosym}
\usepackage[legacycolonsymbols]{mathtools}
\usepackage[ruled,linesnumbered]{algorithm2e}
\providecommand\given{}
\newcommand\SetSymbol[1][]{%
  \nonscript\:#1\vert
  \allowbreak
  \nonscript\:
  \mathopen{}}
\DeclarePairedDelimiterX\Set[1]{\{}{\}}{%
  \renewcommand\given{\SetSymbol[\delimsize]}
  #1}

\DeclarePairedDelimiterXPP\pospart[1]{}{(}{)}{^+}{#1}
\DeclarePairedDelimiterXPP\negpart[1]{}{(}{)}{^-}{#1}

\newcommand\R{\mathbb{R}}

\newcommand\A{\mathcal{A}}

\newcommand\GG{\mathcal{G}}

\newtheorem{thm}{Theorem}[section]
\newtheorem{prop}[thm]{Proposition}
\newtheorem{cor}[thm]{Corollary}
\theoremstyle{definition}

\newtheorem{ex}[thm]{Example}

\DeclareMathOperator*{\optval}{opt\,val}

\DeclareMathOperator*{\Int}{int}

\DeclareMathOperator*{\gr}{gr}
\DeclareMathOperator*{\cl}{cl}
\DeclareMathOperator*{\conv}{conv}

\DeclareMathOperator*{\dom}{dom}

\DeclareMathOperator*{\opt}{opt}
\DeclareMathOperator*{\netflow}{net\,flow}

\makeatletter
\newcommand{\leqnomode}{\tagsleft@true\let\veqno\@@leqno}
\newcommand{\reqnomode}{\tagsleft@false\let\veqno\@@eqno}
\makeatother

\usepackage{booktabs}
\usepackage{tikz}

\usetikzlibrary{positioning}

\definecolor{color1}{rgb}{0,.2,.8}
\definecolor{color2}{rgb}{1,.2,0}
\definecolor{color3}{rgb}{.2,.7,.6}

\title{A Decision-Making Method in Polyhedral Convex Set Optimization}
\author{Andreas L{\"o}hne\thanks{Friedrich Schiller University Jena, Germany, andreas.loehne@uni-jena.de}}

\begin{document}
\maketitle

\begin{abstract} \noindent
Optimization problems with set-valued objective functions arise in contexts such as multi-stage optimization with vector-valued objectives. The aim is to identify an optimizer—a feasible point with an optimal objective value—based on an ordering relation on a family of sets. When faced with multiple optimizers, a decision maker must choose one. Visualizing the values associated with these optimizers could provide a solid basis for decision-making. However, these values are sets, making it challenging to visualize many of them. 
Therefore, we propose a method where an optimizer is selected by designing the respective outcome set through a trial-and-error process. 
In a polyhedral convex setting, we discuss an implementation and prove that an optimizer can be found using this method after a finite number of design steps. We motivate the problem setting and illustrate the process using an example: a two-stage bi-objective network flow problem.

\medskip


\end{abstract}

\section{Introduction}
\label{sec:intro}

Let $F: \mathbb{R}^n \rightrightarrows \mathbb{R}^q$ be a polyhedral convex set-valued map, meaning that there exists a convex polyhedron $\text{gr}\, F$, which will be assumed to be nonempty throughout this paper, and is called the \emph{graph} of $F$,
such that for all $x \in \R^n$, the values of $F$ can be expressed as
$$F(x) = \Set{y \in \R^q \given (x,y) \in \gr F}.$$
The goal is to find some $\bar x \in \R^n$ which is minimal with respect to one of the two partial order relations on the power set of $\R^q$: set inclusion $\subseteq$ and inverse set inclusion $\supseteq$. In the first case, we seek $\bar x \in \R^n$ such that there is no $x \in \R^n$ with $F(x) \subsetneq F(\bar x)$. This case has applications in robust multi-objective optimization \cite{EhrIdeSch14,robust} and game theory with multi-dimensional payoffs \cite{HamLoe18}, but it seems to be much more difficult to handle (in particular, it is inherently non-convex in nature) than the second case, which we focus on in this paper. We define the following minimization problem:
\leqnomode
\begin{gather}\label{eq:p_min}\tag{P$_\text{min}$}
	\min_{x \in \R^n} F(x) \quad \text{ w.r.t. $\;\supseteq$}.
\end{gather}
\reqnomode
A point $\bar x \in \R^n$ is called \emph{minimizer} for \eqref{eq:p_min} if
\begin{equation}\label{eq:optimizer}
	\not\exists x \in \R^n:\; F(x) \supsetneq F(\bar x).
\end{equation}
Problem \eqref{eq:p_min} can be expressed equivalently by the a maximization problem
\leqnomode
\begin{gather}\label{eq:p_max}\tag{P$_\text{max}$}
	\max_{x \in \R^n} F(x) \quad \text{ w.r.t. $\;\subseteq$},
\end{gather}
\reqnomode
and a \emph{maximizer} for \eqref{eq:p_max} is also defined by \eqref{eq:optimizer}. Therefore, it makes sense to remove the optimization direction from the problem formulation and consider the problem
\leqnomode
\begin{gather}\label{eq:p}\tag{P}
	\opt_{x \in \R^n} F(x),
\end{gather}
\reqnomode
which is read as `optimize $F(x)$ subject to $x \in \R^n$'. An \emph{optimizer} for \eqref{eq:p} is a point $\bar x \in \R^n$ satisfying \eqref{eq:optimizer}. Note that $F(\bar x) \neq \emptyset$ for an optimizer $\bar x$, because of the running assumption $\gr F \neq \emptyset$.
 Problem \eqref{eq:p} generalizes both linear programming and multi-objective linear programming (also vector linear programming) in a way that preserves polyhedral convexity. Therefore, Problem \eqref{eq:p} is called \emph{set linear program}. Note that in the literature Problem \eqref{eq:p} in the form of \eqref{eq:p_min} or \eqref{eq:p_max} is usually called \emph{polyhedral convex set optimization problem}.

\begin{ex} \label{ex:lp} Given a linear program 
\leqnomode	
	\begin{gather}\label{eq:lp}\tag{LP}
		\min c^T x \text{ s.t. } A x \geq b,
	\end{gather}
\reqnomode
$\bar x$ is an optimal solution of \eqref{eq:lp} if and only if it is an optimizer of \eqref{eq:p} for 
$$ F(x) \coloneqq \left\{ \begin{array}{cl}
		\Set{c^T x} +\R_+ & \text{ if } A x \geq b\\
		\emptyset.        & \text{ otherwise.}
		\end{array}\right.
$$
A maximization problem can be handled similarly by replacing the non-negative real numbers $\R_+$ with the nonpositive real numbers $\R_- \coloneqq -\R_+$. The relationship to vector linear programming can be seen similarly by replacing $c^T$ by a suitable objective matrix and $\R_+$ by a polyhedral convex ordering cone, for details, see, e.g., \cite[Section 4]{Loe_natcon}.
\end{ex}

For the set linear program \eqref{eq:p}, we consider two decision makers, DM1 and DM2. DM1's task is to choose an optimizer $\bar x$ of $\eqref{eq:p}$,
while DM2 (typically later) chooses some vector $\bar y \in F(\bar x)$. A typical application is a  multi-stage
optimization problem with a vector-valued objective function $f:\R^n \times \R^k \to \R^q$. In the first stage, DM1 chooses some
$x \in \R^n$ and later DM2 chooses some $u \in \R^k$, which yields an objective value $y = f(x,u) \in \R^q$.
This situation can be described by a set-valued mapping $F : \R^n \rightrightarrows \R^q$, defined as 
$$ F(x) \coloneqq \Set{ f(x,u) \given u \in \R^k}.$$
When DM2 selects some $y \in F(x)$, there exists a corresponding $u \in \mathbb{R}^k$ such that $y = f(x,u)$. Although $u$ may not necessarily be uniquely defined, it is indirectly chosen by DM2.
Note that only if an optimizer is chosen by DM1 can DM2 have the maximum flexibility in selecting $y \in F(x)$. This means that DM1 has a ``preference for flexibility'', see \cite{Kreps79} and the discussion in \cite{HamLoe20, HamLoe24}.

Set linear programming (also called polyhedral convex set optimization) has applications in set-valued coherent risk measures introduced in \cite{JouMedTou04} and extended in \cite{HamHey10, HamHeyRud11}. The potential role of optimizers and a preference for flexibility in this field is sketched in a simple example in \cite{Loe24}. Another field of application, where optimizers and the flexibility interpretation are central, is multi-objective multi-stage stochastic linear programming \cite{HamLoe24}. Below in Section \ref{sec:ex} we introduce and discuss a further potential application in the field of energy or transportation network design.

In several papers on polyhedral convex set optimization (see, e.g., \cite{LoeSch13, HeyLoe, Loe24, HamLoe24}), a polyhedral convex ordering cone $C \subseteq \R^q$ is considered. Similarly, in more general settings, set relations involving an ordering cone \cite{Kuroiwa98} are used instead of set inclusion.
This cone typically represents the preferences of the decision makers. Although important to the decision makers, it does not play any role in the decision-support process developed in this paper. We maintain generality by assuming that
\begin{equation}\label{eq:FC}
	\forall x \in \R^n:\; F(x) = F(x) + C.
\end{equation}
 Moreover, it is not necessary to explicitly mention any linear constraints in the formulation of Problem \eqref{eq:p} as they can be subsumed in the polyhedral convex objective function. Indeed, if $S \subseteq \mathbb{R}^n$ is a polyhedral convex feasible set, intersecting the graph of $F$ with $S \times \mathbb{R}^q$ yields a new polyhedral convex set-valued map and an equivalent problem free of constraints.

The literature presents two main application-relevant approaches for optimization problems with set-valued objective mappings (e.g., \cite{HamHeyLoeRudSch15,KhaTamZal15}): the \emph{set approach}, which seeks to compute all or some optimizers without requiring infimum or supremum attainment, and the \emph{complete lattice approach}, which aims to find optimizers that attain the infimum or supremum. The present approach shares properties of both: In principle, we can compute all optimizers, but in practice, we compute only one, which is preferred by the decision maker. On the other hand, we do not require infimum or supremum attainment, though the infimum or supremum (called the optimal value here) is still part of the information presented to the decision maker.

A solution method for polyhedral convex set optimization problems based on the complete lattice approach has been developed in \cite{LoeSch13, Loe24}. Other contributions consider more general problem classes. In \cite{Jahn15}, optimizers (minimizers) are obtained from an associated vector optimization problem with infinitely many (or, in the polyhedral case, many) objectives. The recent works \cite{EicQuiRoc22,EicRoc23} provide methods for computing all weakly minimal solutions (or weak optimizers, in our terminology) for set optimization problems by solving an associated vector optimization problem. In our setting, a \emph{weak optimizer} $\bar{x} \in \R^n$ of \eqref{eq:p} is defined by the condition:
\[
    \not\exists x \in \R^n : \Int F(x) \supseteq F(\bar{x}).
\]
Clearly, a weak optimizer is not necessarily an optimizer. This shows that a weak optimizer is generally not the most flexible decision, as discussed above.

We next present a simple example that illustrates a drawback of complete-lattice-type solutions \cite{HeyLoe11,LoeSch13,HeyLoe,Loe24} for polyhedral convex set optimization problems within the context of the flexibility interpretation outlined earlier. Let us first recall this solution concept.
The concept is based on the observation that the image space of the objective function forms a complete lattice, meaning that every subset has both an infimum and a supremum. In our setting, the family $\GG$ of all closed convex subsets of $\R^q$ serves as a suitable image space. For a subset $\A$ of $\GG$, partially ordered by $\supseteq$, we have 
$$ \inf \A = \cl \conv \bigcup_{A \in \A} A \qquad \text{and} \qquad \sup_{A \in \A} \A = \bigcap_{A \in \A} A,$$
where $\cl\conv Q$ denotes the closed convex hull of a set $Q \subseteq \R^q$.
The notion of infimum is appropriate for the minimization problem \eqref{eq:p_min}, while for \eqref{eq:p_max} the same expression should be interpreted as a supremum, i.e.,
\[
\sup \A = \cl \conv \bigcup_{A \in \A} A.
\]
However, this distinction is not necessary here. To unify the treatment, we define the \emph{optimal value} of \( F \) subject to \( X \subseteq \R^n \) as
\[
\optval_{x \in X} \coloneqq \cl\conv \bigcup_{x \in X} F(x).
\]
For $X = \R^n$ (and similarly for polyhedral convex sets), we have
$$ \optval_{x \in \R^n} F(x) = \bigcup_{x \in \R^n} F(x),$$
i.e., the closed convex hull can be omitted (see, e.g., \cite[Proposition 3]{HeyLoe}). 

To simplify the exposition, we recall the solution concept only for the specific case of \emph{bounded} polyhedral convex set optimization problems. 
This means we assume that there is a finite set $\bar{X} \subseteq \R^n$ where the optimal value of $F$ is attained, that is,
\begin{equation}\label{eq:inf_att}
	\optval_{x \in \bar X} F(x) = \optval_{x \in \R^n} F(x).
\end{equation} 
A finite set $\bar{X} \subseteq \R^n$ of optimizers of \eqref{eq:p} satisfying \eqref{eq:inf_att} is called a \emph{solution} of \eqref{eq:p} \cite{HeyLoe11,LoeSch13}.
In the non-bounded case, finite attainment is still possible, but directions in $\R^n$ and the recession function of $F$ need to be involved in the solution definition \cite{HeyLoe,Loe24}. Note that no such boundedness assumption is required for the methods introduced in this paper.

\begin{ex}\label{ex:1}
	Consider the set $A_1 = \Set{(1,0)^T}+\R^2_+$, $A_2 \Set{(0,1)^T}+\R^2_+$, $A_3= \conv\Set{(1.05,0.05)^T,
	(0.05,1.05)^T}+\R^2_+$, depicted in Figure \ref{fig:1}, and let the set-valued map
	$F:\R^3 \rightrightarrows \R^2$ be given as
	$$ F(x) = \left\{ \begin{array}{cl}
	x_1 A_1 + x_2 A_2 + x_3 A_3 & \text{ if } x_1,x_2,x_3 \geq 0,\; x_1+x_2+x_3=1\\
	\emptyset                   &\text{ otherwise. }
	\end{array}\right.$$
	Then we have
	$$ \optval_{x \in \R^n} F(x) = \conv\Set{(1,0)^T, (0,1)^T}+\R^2_+ .$$
	All unit vectors $e^1, e^2, e^3$ are optimizers, and have the values $F(e^i) = A_i$, $i = 1, 2, 3$.
	The set $\bar{X} = \Set{e^1, e^2}$ is a solution;
	its elements $e^1$ and $e^2$, with values $A_1$ and $A_2$,
	do not provide any flexibility for DM2. The optimizer $e^3$,
	which provides much flexibility, is not necessarily contained in a solution. 
	Thus, this optimizer is not available for the decision maker if a solution as defined above is used.

	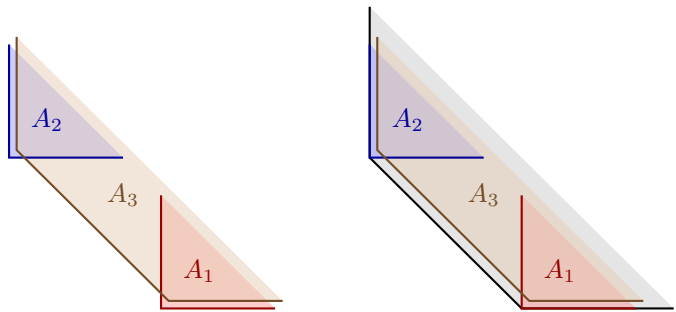
\begin{figure}
		\begin{center}
			\begin{tikzpicture}
				\coordinate (A1) at (2.5, 0.5);
				\coordinate (A11) at (2, 1.5);
				\coordinate (A12) at (2, 0);
				\coordinate (A13) at (3.5, 0);
				\coordinate (B1) at (0.5, 2.5);
				\coordinate (B11) at (1.5, 2);
				\coordinate (B12) at (0, 2);
				\coordinate (B13) at (0, 3.5);
				\coordinate (C1) at (1.5, 1.5);
				\coordinate (C11) at (3.6, 0.1);
				\coordinate (C12) at (2.1, 0.1);
				\coordinate (C13) at (0.1, 2.1);
				\coordinate (C14) at (0.1, 3.6);
				\coordinate (AA13) at (4, 0);
				\coordinate (BB13) at (0, 4);

				\fill[red!40, opacity=.5] (A11) -- (A12) -- (A13) -- cycle;
				\fill[blue!40, opacity=.5] (B11) -- (B12) -- (B13) -- cycle;
				\fill[brown!40, opacity=.5] (C11) -- (C12) -- (C13) -- (C14) -- cycle;
				\draw[red!60!black, thick] (A11) -- (A12) -- (A13);
				\draw[blue!60!black, thick] (B11) -- (B12) -- (B13);
				\draw[brown!60!black, thick] (C11) -- (C12) -- (C13) -- (C14);
				\node[red!60!black]  at (A1) [] {$A_1$};
				\node[blue!60!black] at (B1) [] {$A_2$};
				\node[brown!60!black] at (C1) [] {$A_3$};
			\end{tikzpicture}\hspace{1cm}
			\begin{tikzpicture}
				\coordinate (A1) at (2.5, 0.5);
				\coordinate (A11) at (2, 1.5);
				\coordinate (A12) at (2, 0);
				\coordinate (A13) at (3.5, 0);
				\coordinate (B1) at (0.5, 2.5);
				\coordinate (B11) at (1.5, 2);
				\coordinate (B12) at (0, 2);
				\coordinate (B13) at (0, 3.5);
				\coordinate (C1) at (1.5, 1.5);
				\coordinate (C11) at (3.6, 0.1);
				\coordinate (C12) at (2.1, 0.1);
				\coordinate (C13) at (0.1, 2.1);
				\coordinate (C14) at (0.1, 3.6);

				\fill[gray!40, opacity=.5] (AA13) -- (A12) -- (B12) -- (BB13) -- cycle;
			
				\fill[red!40, opacity=.5] (A11) -- (A12) -- (A13) -- cycle;
				\fill[blue!40, opacity=.5] (B11) -- (B12) -- (B13) -- cycle;
				\fill[brown!40, opacity=.5] (C11) -- (C12) -- (C13) -- (C14) -- cycle;
			
				\draw[thick] (AA13) -- (A12) -- (B12) -- (BB13);
			
				\draw[red!60!black, thick] (A11) -- (A12) -- (A13);
				\draw[blue!60!black, thick] (B11) -- (B12) -- (B13);			
				\draw[brown!60!black, thick] (C11) -- (C12) -- (C13) -- (C14);
			
				\node[red!60!black]  at (A1) [] {$A_1$};
				\node[blue!60!black] at (B1) [] {$A_2$};
				\node[brown!60!black] at (C1) [] {$A_3$};
			\end{tikzpicture}
		\end{center}
		\caption{Left: the three sets defining $F$ in Example \ref{ex:1}.
		Right: the optimal value of $F$ is additionally shown in the rear. The set $A_3$ does not `contribute' to the optimal value.
}\label{fig:1}
	\end{figure}
\end{ex}

Despite this drawback of a solution, in many applications, a solution can provide a good first overview for the decision maker. Therefore, it might make sense to compute a solution before starting the optimizer design process. On the other hand, the method we propose here can also be used to compute solutions. It is preferable if only a part of the solution is interesting for the decision maker as an overview. The design process can also be started directly, allowing the decision maker to explore a finite number of preferred options individually.

We would like to add that computing all optimizers is often impractical because there are usually infinitely many. Presenting those with certain extremal properties, such as those that cannot be expressed as a convex combination of others, is also impractical due to the sheer number of such optimizers in typical problem instances. 

The paper is structured as follows: In Section \ref{sec:od}, we begin with an informal description of the proposed method. Section \ref{sec:theory} presents the core algorithm for an abstract problem setting, along with a demonstration of its key properties, particularly its correctness. In Section \ref{sec:exist}, we address the polyhedral convex problem setting and characterize the existence of optimizers. The finiteness of the method is discussed in Section \ref{sec:finite}. An example that motivates the problem class and illustrates the method is provided in Section \ref{sec:ex}. Finally, implementation issues are discussed in Section \ref{sec:impl}.

\section{Informal description of optimizer design}\label{sec:od}

With the method introduced in this paper, one is able to explore the set of optimizers. The decision maker's task is to successively select points $y^1, \dots, y^k \in \R^q$. The goal is to find an optimizer $x$ of Problem \eqref{eq:p} such that all these points $y^i$ belong to $F(x)$. 

The first point $y^1$ can be chosen by the decision maker arbitrarily from the optimal value of Problem \eqref{eq:p}, that is, from the set
$$ Y_0 \coloneqq \bigcup_{x \in \R^n} F(x)= \Set{y \in \R^q \given \exists x \in \R^n:\; y \in F(x)}.$$
It is important that the decision maker can see all her options when selecting $y^1$ from the set $Y_0$, in order to be able to take into account her preferences. Therefore, the set $Y_0$ must be `visualized' in some way.

If the second point $y^2$ is also chosen arbitrarily from the set $Y_0$, then there does not necessarily exist an $x$ such that $F(x)$ contains both $y^1$ and $y^2$. Therefore, $y^2$ must be chosen from the optimal value $Y_1$ of a modified problem, where the constraint $y^1 \in F(x)$ is added:
$$ Y_1 \coloneqq \Set{y \in \R^q \given \exists x \in \R^n:\; y \in F(x),\; y^1 \in F(x)}.$$
Again, this set needs to be visualized for the decision maker to fully understand all available options and effectively consider her preferences.

At this point, it may happen that the decision maker is not satisfied with the set of available options, which, of course, depends on the choice of $y^1$. Thus, the selection of $y^1$ may need to be revised. Once the decision maker is satisfied with the available options $Y_1$, the point $y^2$ can be selected from $Y_1$. The constraint $y^2 \in F(x)$ needs to be added, and the decision maker then studies the set
$$ Y_2 \coloneqq \Set{y \in \R^q \given \exists x \in \R^n:\; y \in F(x),\; y^1 \in F(x),\; y^2 \in F(x)}.$$
If necessary, the decision maker can now revise the choice of $y^1$ or $y^2$. If satisfied with the available options $Y_2$, she continues to select some point $y^3$ from $Y_2$. This procedure continues in this manner.

The following observations can be made: Adding constraints leads to the same or smaller optimal values. This means that we have
$$ Y_0 \supseteq Y_1 \supseteq Y_2 \supseteq \dots $$
The set $Y_0$ is nonempty whenever the graph of $F$ is nonempty, which is assumed. By induction, it follows that all the constructed sets $Y_i$ are nonempty. Indeed, if $y^{i+1} \in Y_i$, then there exists $x \in \mathbb{R}^n$ such that $y^{i+1} \in F(x)$, $y^1 \in F(x)$, $\dots$, $y^i \in F(x)$, which implies that $Y_{i+1} \neq \emptyset$.

The decision maker's options $Y_i$ are not necessarily realized by a single feasible point $x \in \mathbb{R}^n$, that is,
$$ \not\exists x \in \mathbb{R}^n \text{ such that } F(x) \supseteq Y_i. $$
In this case, the procedure continues. Otherwise, it can be terminated: If for some $k$ there exists $x \in \mathbb{R}^n$ such that $F(x) \supseteq Y_k$, then this $x$ is an optimizer for \eqref{eq:p} that contains the points $y^1, \dots, y^k$ selected by the decision maker.

In the subsequent sections, we will formalize this procedure and prove its correctness and finiteness under certain assumptions.

\section{Theoretical framework and core algorithm}\label{sec:theory}

Let $F: \R^n \rightrightarrows \R^q$ be an arbitrary set-valued mapping. The \emph{domain} of $F$ is defined as
$$ \dom F \coloneqq \Set{x \in \R^n \given F(x) \neq \emptyset}.$$
By the running assumption $\gr F \neq \emptyset$, the domain of $F$ is always nonempty.

The \emph{value function} of $F$ is defined as
\[ 
v_F:2^{\R^q} \to 2^{\R^q},\quad v_F(Y) \coloneqq \Set{z \in \R^q \given \exists x \in \R^n:\; z \in F(x),\; Y \subseteq F(x)}.
\]
It can be expressed equivalently as
\[
   v_F(Y) = \bigcup_{x \in \R^n,\; F(x) \supseteq Y} F(x).
\]
This shows that $v_F(Y)$ is the optimal value of the problem to minimize $F$ subject to the constraint $F(x) \supseteq Y$. 

A subset $Y$ of $\R^q$ is called \emph{fixed point} of the value function $v_F$ if
\begin{equation}\label{eq:fixed_point}
	Y=v_F(Y).
\end{equation}
\begin{prop} \label{prop:vf_fixed_point}
Let $F: \R^n \rightrightarrows \R^q$ be a set-valued map (with nonempty domain) and let $Y \subseteq \R^q$. The following is equivalent:
\begin{enumerate}[(i)]
	\item $Y$ is a fixed point of $v_F$,
	\item There exists an optimizer $\bar x \in \R^n$ of $F$ with $Y=F(\bar x)$.
\end{enumerate}
\end{prop}
\begin{proof} 	
Let $\bar{x} \in \R^n$ be an optimizer of $F$. Since $\dom F \neq \emptyset$, an optimizer $\bar x$ satisfies $F(\bar{x}) \neq \emptyset$, and there is no $x \in \R^n$ such that $F(x) \supsetneq F(\bar{x})$. From the definition of the value function, we get $F(\bar{x}) \subseteq v_F(F(\bar{x}))$. Assume $Y = F(\bar x)$ is not a fixed point of $v_F$, that is, the latter inclusion is strict. Then there exists $z \in v_F(F(\bar{x}))$ such that $z \notin F(\bar{x})$. By the definition of the value function, there is $x \in \R^n$ such that $z \in F(x)$ and $F(\bar{x}) \subseteq F(x)$. Together, this implies $F(x) \supsetneq F(\bar{x})$, a contradiction.

Conversely, let $Y$ be a fixed point of $v_F$, that is, $Y= v_F(Y)$. Since $F$ has nonempty domain, we have $v_F(\emptyset) = \Set{y \given \exists x \in \R^n:\; y \in F(x)} \neq \emptyset$, thus $Y = v_F(Y) \neq \emptyset$. By the definition of the value function there exists some $\bar x \in \R^n$ such that $Y \subseteq F(\bar x)$. Assume $\bar x$ is not an optimizer of $F$. Then there is some $x \in \R^n$ with $F(x) \supsetneq F(\bar x)$. Using the definition of the value function this implies $v_F(Y) \supsetneq Y$.
\end{proof}
We say that $Y \subseteq \R^q$ \emph{generates a fixed point} of $v_F$ if there is a fixed point $Z \subseteq \R^q$ of $v_F$ satisfying
$$ v_F(Y) = v_F(Z).$$
The following statement is an immediate consequence of Proposition \ref{prop:vf_fixed_point}.

\begin{cor} \label{cor:vf_fixed_point}
Let $F: \R^n \rightrightarrows \R^q$ be a set-valued map (with nonempty domain) and let $Y \subseteq \R^q$. The following is equivalent:
\begin{enumerate}[(i)]
	\item $Y$ generates a fixed point of $v_F$,
	\item There exists an optimizer $\bar x \in \R^n$ of $F$ with $v_F(Y) = F(\bar x)$.
\end{enumerate}	
\end{cor}

The \emph{domain} of the value function $v_F$ is the family of sets 
$$\dom v_F \coloneqq \Set{Y \subseteq \R^q \given v_F(Y) \neq \emptyset}.$$
The next proposition shows that one inclusion in \eqref{eq:fixed_point} is closely related to the domain of $v_F$.
\begin{prop} \label{prop:vf_domain}
Let $F: \R^n \rightrightarrows \R^q$ be a set-valued map (with nonempty domain). Then, the domain of the value function $v_F$ is
\[
	\dom v_F = \Set {Y \in 2^{\R^q} \given Y \subseteq v_F(Y)}.
\]
\end{prop}
\begin{proof}	
Let \( Y \in \dom v_F \), i.e., \( v_F(Y) \neq \emptyset \). By the definition of the value function, there exist \( z \in \R^q \) and \( x \in \R^n \) such that \( z \in F(x) \) and \( Y \subseteq F(x) \). Thus, again by the definition of $v_F$, every \( y \in Y \) belongs to \( v_F(Y) \).
Conversely, let \( Y \in 2^{\R^q} \) such that \( Y \subseteq v_F(Y) \). If \( Y \) is nonempty, then \( v_F(Y) \) is also nonempty. In the remaining case, when \( Y = \emptyset \), we have
\[
v_F(\emptyset) = \Set{y \in \R^q \given \exists x \in \R^n:\; y \in F(x)}.
\]
This set is nonempty since the domain of \( F \) was assumed to be nonempty.
\end{proof}
We continue by pointing out a monotonicity property of the value function.
\begin{prop} \label{prop:vf_monotone}
The value function of a set-valued map \( F: \R^n \rightrightarrows \R^q \) is inclusion reversing:
	$$ Y \subseteq Z \implies v_F(Y) \supseteq v_F(Z).$$ 	
\end{prop}
\begin{proof}
Let $y \in v_F(Z)$. Then there exists $x \in \R^n$ such that $\Set{y} \cup Z \subseteq F(x)$. The inclusion $Y \subseteq Z$ implies $\Set{y} \cup Y \subseteq F(x)$, whence $y \in v_F(Y)$.
\end{proof}
The results we have proven suggest the idea of enlarging the set 
$Y$ iteratively until, if possible, a fixed point of the value function is obtained. Such a construction requires that we do not leave the domain of the value function. This is guaranteed in the situation described in the next proposition.

\begin{prop} \label{prop:vf_nonempty}
	Let \( F: \R^n \rightrightarrows \R^q \) be a set-valued map, \( Y \subseteq \R^q \), and $z \in v_F(Y)$. Then
	$$ v_F(Y\cup\Set{z})\neq \emptyset.$$
\end{prop}
\begin{proof}
	Since $z \in v_F(Y)$ there exists $x \in \R^n$ such that $z \in F(x)$ and $Y \subseteq F(x)$. Thus we have $Y \cup \Set{z} \subseteq F(x)$, i.e., 
	every element in $Y \cup \Set{z}$ belongs to $v_F(Y \cup \Set{z})$.
\end{proof}

Algorithm \ref{alg:1} uses the mentioned ideas and results to construct an optimizer of $F$. We will prove the correctness of the algorithm. In order to show that the algorithm terminates after finitely many steps, additional structure and further assumptions are necessary. Such a setting will be discussed in the next two sections.

\begin{algorithm}[ht]
\DontPrintSemicolon
\SetKwInOut{Input}{input}\SetKwInOut{Output}{output}
\Input{$F:\R^n \rightrightarrows \R^q$ such that $\dom F \neq \emptyset$}
\Output{minimizer $x$ of $F$}
\Begin{
	$Y \leftarrow \emptyset$\;
	\While{$Y$ does not generate a fixed point of $v_F$}
	{
		decision-making step: choose some $y \in v_F(Y)$\; \label{alg:max:line:4}
		$Y\leftarrow Y \cup \Set{y}$;
	}
	determine some $x \in \R^n$ with  $v_F(Y) = F(x)$
}
\caption{Principle of optimizer design}
\label{alg:1}
\end{algorithm}

\begin{prop} \label{prop:correct}
	Algorithm \ref{alg:1} is correct: If it terminates after finitely many steps then it returns an optimizer $x$ of $F$.
\end{prop}
\begin{proof}
	The assumption $\dom F \neq \emptyset$ implies that $v_F(\emptyset) \neq \emptyset$. The set $v_F(Y)$ remains nonempty by Proposition \ref{prop:vf_nonempty}. After the loop (if it is left), $Y$ generates a fixed point of $v_F$. By Corollary \ref{cor:vf_fixed_point} there exists an optimizer $\bar x$ of $F$ such that $v_F(Y) = F(\bar x)$. Then, every $x \in \R^n$ with $v_F(Y) = F(x)$ is also an optimizer of $F$.
\end{proof}

It remains open at this point how the stopping criterion and the computation of $\bar x$ in Algorithm \ref{alg:1} can be realized. In the subsequent sections we show how this is possible for a special setting. 

Finally, let us discuss the potential role of a decision maker in Algorithm~\ref{alg:1}. The points $y$ in the decision-making step can be chosen arbitrarily, which means this task could be performed by a decision maker. Thus, the decision maker can control the construction of an optimizer by deciding which points are added to $Y$. If the algorithm terminates, all points in $Y$ belong to $F(\bar{x})$ for the constructed point $\bar{x}$.

\section{Existence of optimizers}\label{sec:exist}

The algorithm introduced in the previous section computes an optimizer whenever it terminates. Thus, the existence of an optimizer is necessary for the finiteness of the algorithm. In this section, we recall some recent results on the existence of optimizers. In particular, we consider a simplified setting where no explicit order cone is involved.

The \emph{recession function} of a polyhedral convex set-valued map $F:\R^n \rightrightarrows \R^q$ with nonempty domain is defined as
$$ G:\R^n \rightrightarrows \R^q,\quad G(x)= \Set{y \in \R^q \given (x,y) \in 0^+ \gr F},$$
where $0^+ Q$ denotes the recession cone of a convex set $Q$. This means the graph of $G$ is just the recession cone of the graph of $F$.
For $x \in \dom F$ we have (e.g., \cite[Proposition 2]{HeyLoe})
\begin{equation}\label{eq:recFx}
	0^+F(x) = G(0),
\end{equation}
i.e., all nonempty values $F(x)$ have the same recession cone. 
Moreover, the existence of optimizers can be characterized by the recession function.
\begin{prop}\label{prop:ex_opt1}
	Let  $F:\R^n \rightrightarrows \R^q$ be a polyhedral convex set-valued map (with nonempty domain). The following is equivalent:
	\begin{enumerate}[(i)]
		\item There exists an optimizer $\bar x \in \R^n$ of $F$.
		\item For all $\bar y \in \optval_{x \in \R^n} F(x)$ there exists an optimizer $\bar x \in \R^n$ of $F$ such that $\bar y \in F(\bar x)$.
		\item The origin $0 \in \R^n$ is an optimizer of $G$.
	\end{enumerate}
\end{prop} 
\begin{proof}
	This is a special case of \cite[Proposition 3.2]{Loe} for the ordering cone $C=\Set{0}$.
\end{proof}
The \emph{natural ordering cone} \cite{Loe_natcon} of a polyhedral convex set-valued map $F:\R^n \rightrightarrows \R^q$ with nonempty domain, and with the recession function denoted by $G$, is defined as
$$ K\coloneqq K_F \coloneqq \Set{y \in \R^q \given \exists x \in \R^n:\; y \in G(x),\, 0 \in G(x)}.$$
We have from \cite[Proposition 3.2]{Loe_natcon} that
\begin{equation}\label{eq:noc}
	G(0) \subseteq K \subseteq \bigcup_{x \in \R^n} G(x),
\end{equation}
where both inclusions can be strict, see \cite[Example 3.3]{Loe_natcon}. The natural ordering cone can be used to characterize the existence of optimizers.

\begin{prop}\label{prop:ex_opt2}
	Let  $F:\R^n \rightrightarrows \R^q$ be a polyhedral convex set-valued map with nonempty domain. Then the three equivalent statements in Proposition \ref{prop:ex_opt1} are satisfied if and only if 
	\begin{equation}\label{eq:opt_cond}
		G(0) = K.
	\end{equation}
\end{prop}
\begin{proof}
	Let \eqref{eq:opt_cond} be violated. Because of \eqref{eq:noc} there exists some $y \in K \setminus G(0)$. By the definition of $K$, there exists $x \in \R^n$ such that $y \in G(x)$ and $0 \in G(x)$. From \eqref{eq:recFx} (applied to $G$) we get $G(x) = G(x) + G(0)$. Thus $0 \in G(x)$ implies $G(0) \subseteq G(x)$. Since $y \not\in G(0)$, this inclusion is strict. Thus, $0$ is not an optimizer of $G$.
	
Conversely, assume that $0$ is not an optimizer of $G$. Then there is some $x \in \R^n$ such that $G(x)\supsetneq G(0)$. From $0 \in G(0)$ we get $0 \in G(x)$. Taking some $y \in G(x)\setminus G(0)$, we have $y \in K \setminus G(0)$, i.e., \eqref{eq:opt_cond} is violated. 
\end{proof}

Note that Proposition \ref{prop:ex_opt2} can also be obtained as a special case of \cite[Corollary 3.5]{Loe_natcon}.

\section{Finiteness of the algorithm}\label{sec:finite}

In this section we provide sufficient conditions under which the optimizer design method of Algorithm \ref{alg:1} in Section \ref{sec:theory} is finite and thus outputs an optimizer. First we assume polyhedral convexity of the objective mapping $F$. Secondly, we suppose that an optimizer exists, which can be expressed by Condition \eqref{eq:opt_cond} in Section \ref{sec:exist}. The third and last assumption is a requirement to the decision-making process. 
The decision-making steps---that is, the choice of \( y \in v_F(Y) \) in line \ref{alg:max:line:4} of Algorithm \ref{alg:1}---are required to be \emph{qualified} in the following sense, where finitely many exceptions are possible:
\begin{quote}
	\emph{The chosen point $y$ must belong to a minimal (with respect to $\subseteq$) face $S$ of the convex polyhedron $v_F(Y)$ such that $S \cap Y = \emptyset$.}
\end{quote}
The easiest way to illustrate this condition is to consider the particular case where \( v_F(\emptyset) = \optval_{x \in \R^n} F(x)\) has a vertex. In this case, the minimal faces of the sets \( v_F(Y) \) are their vertices. Thus, the decision maker only needs to choose a vertex of \( v_F(Y) \) that has not been selected before.

Before we prove finiteness of the algorithm, we show some properties of the value function.
A \emph{P-representation} (where P stands for `projection') of a polyhedral convex set $P$ is a representation of the form
$$ P = \{ y \in \mathbb{R}^q \mid \exists x \in \mathbb{R}^n : A x + B y \geq b \} $$
for matrices $A \in \R^{m \times n}$, $B \in \R^{m \times q}$ and a vector $b \in \R^m$.
If $P$ is nonempty, its recession cone can be expressed as
$$ 0^+ P = \{ y \in \mathbb{R}^q \mid \exists x \in \mathbb{R}^n : A x + B y \geq 0 \}, $$
see, e.g., \cite{CirLoeWei19}.

\begin{prop} \label{prop:polyh} Let $F : \R^n \rightrightarrows \R^q$ be a polyhedral convex set-valued mapping. Then for every finite subset $Y \subseteq \R^q$, the value $v_F(Y)$ of the value function $v_F$ is a convex polyhedron. In case $v_F(Y)$ is nonempty, its recession cone the natural ordering cone $K$ of $F$, that is, $0^+ v_F(Y) = K$.
\end{prop}
\begin{proof} 
	For $Y=\Set{y^1,\dots,y^k}$, we have 
	$$v_F(Y)=\Set{y\in \R^q \given \exists x \in \R^n:\; (x,y), (x,y^1), \dots, (x,y^k) \in \gr F},$$
	 which is a P-represented polyhedral convex set, whenever the graph of $F$ is polyhedral convex. If the set $v_F(y)$ is nonempty, its recession cone is obtained as
	\begin{align*}
		0^+ v_F(Y) &= \Set{y\in \R^q \given \exists x \in \R^n:\; (x,y), (x,0), \dots, (x,0) \in 0^+ \gr F} \\
		           &= \Set{y\in \R^q \given \exists x \in \R^n:\; y \in G(x),\, 0 \in G(x)}. 
	\end{align*}
	Thus, we have $0^+ v_F(Y) = K$.
\end{proof}

Now we are ready to prove the main result.

\begin{thm}\label{thm:1}
	Algorithm \ref{alg:1} is correct and finite if the following three conditions are satisfied:
	\begin{enumerate}[(i)]
		\item Polyhedral convex problem setting: $F$ is polyhedral convex with nonempty domain.
		\item Existence of optimizers: The natural ordering cone $K$ equals $G(0)$.
		\item Qualified decision-making: The decision-making steps in line \ref{alg:max:line:4} are qualified with at most finitely many exceptions.
	\end{enumerate}
\end{thm}
\begin{proof} Concerning the correctness of the algorithm, in view of Proposition \ref{prop:correct}, it remains to show that a qualified decision-making step is always possible within the algorithm. The decision making step in line \ref{alg:max:line:4} is executed only if $Y$ does not generate a fixed point of $v_F$. Thus, there is no fixed point $Z$ of $v_F$ with $v_F(Y)=v_F(Z)$. 	
Define $Z\coloneqq \conv Y + G(0)$. Then we have $v_F(Y)=v_F(Z)$, which follows from the definition of $v_F$ and the fact that $G(0)$ is the recession cone of $F(x)$ for every $x\in \dom F$. We conclude that $Z$ is not a fixed point of $v_F$. Since $v_F(Y)$ constructed in the algorithm is always nonempty, Proposition \ref{prop:vf_domain} yields the inclusion $Z \subseteq v_F(Z)$. Together we obtain $Z \subsetneq v_F(Y)$, i.e., there exists $y \in v_F(Y) \setminus Z$. The assumption $K= G(0)$ yields that $v_F(Y)$ and $Z$ both have the same recession cone $G(0)$. Therefore, without loss of generality, $y$ can assumed to belong to a minimal face $S$ of $v_F(Y)$. We show that $S \cap Y = \emptyset$. Assuming that there is some $v \in S \cap Y$. Then there is some $l$ in the lineality space $L=-G(0)\cap G(0)$ of $v_F(Y)$ such that $y = v + l$. This implies
$y = v + l \in Y + L \in \conv Y + G(0) = Z$, which is a contradiction. This proves the correctness of the algorithm under the specification in (iii).

To show that the algorithm is finite, it remains to prove the following two statements:
\begin{enumerate}[(a)]
	\item If $S$ a minimal face of $v_F(Y^{(k)})$ with $S \cap Y^{(k)} \neq \emptyset$ in some iteration $k$ of the algorithm, then the same $S$ is also a minimal face of $v_F(Y^{(l)})$ with $S \cap Y^{(l)} \neq \emptyset$ in any later iteration $l>k$ of the algorithm.
	\item There is an upper bound $\kappa$ on the number of minimal faces of $v_F(Y^{(k)})$, which is independent of $k$.
\end{enumerate} 
Condition (a) means that the same minimal face cannot be used twice in a qualified decision making step. As a consequence, the second statement (b) means that at most $\kappa$ qualified decision making steps are possible. Thus the algorithm must be finite. 

To prove (a), let $S$ be a face of $v_F(Y^{(k)})$. Since $v_F(Y^{(k)})$ is a convex polyhedron by Proposition \ref{prop:polyh}, $S$ is an exposed face, i.e., $S$ is the set of optimal solutions of some linear program with feasible set $v_F(Y^{(k)})$. For $l>k$ we have $Y^{(l)} \supseteq Y^{(k)}$ and, by Proposition \ref{prop:vf_monotone}, $v_F(Y^{(l)}) \subseteq v_F(Y^{(k)})$. Since $S$ is a minimal face of $v_F(Y^{(k)})$ and, by assumption, $S \cap Y^{(k)} \neq \emptyset$, it is of the form $S = \Set{y} + L$, where  $y \in Y^{(k)}$ and $L$ being the lineality space of $v_F(Y^{(k)})$. We conclude that $y \in Y^{(l)}$. Since $L$ is also the lineality space of $v_F(Y^{(l)})$, and $Y^{(l)} \subseteq v_F(Y^{(l)})$ by Proposition \ref{prop:vf_nonempty}, we conclude that $S \subseteq v_F(Y^{(l)})$. Thus $S$ remains the set of optimal solutions of the mentioned linear program if its feasible set $v_F(Y^{(k)})$ is replaced by $v_F(Y^{(l)}) \subseteq v_F(Y^{(k)})$. Thus $S$ is also a face of $v_F(Y^{(l)})$. Because of its form $S = \Set{y} + L$, it is a minimal one.

To prove (b), we note that an H-representation of the graph of $F$ exists. It can be expressed as
$$ \gr F = \Set{(x,y) \in \R^n \times \R^q \given A x + B y \leq b}.$$
For $Y^{(k)}=\Set{y^1,\dots,y^k}$, this provides a representation
$$ v_F(Y^{(k)}) = \Set{y \in \R^q \given \exists x \in \R^n:\, A x + B y \leq b, A x \leq b - B y^1, \dots, A x \leq b - B y^k }.$$
Let $\bar b$ be the component-wise minimum of the vectors $b - B y^1, \dots, b - B y^k$. Then we can write
$$ v_F(Y^{(k)}) = \Set{y \in \R^q \given \exists x \in \R^n:\, A x + B y \leq b, A x \leq \bar b}.$$
Let $m$ denote the number of rows of $A$, then the latter P-representation of $v_F(Y^{(k)})$ has $2m$ inequalities. This number is independent of $k$, i.e., it is the same in each iteration of the algorithm. Fourier-Motzkin elimination can be used to compute an $H$-representation of $v_F(Y^{(k)})$, which has at most $\kappa \coloneqq 4((2m)/4)^{2^n}$ inequalities. Then $\kappa$ is an upper bound for the number of minimal faces $S$ which can be used in qualified decision making steps.
\end{proof}

\section{Example: network supply capacity design} \label{sec:ex}

In this section, we discuss a potential application of polyhedral convex set optimization in order to motivate the problem setting and to illustrate the decision-making method. 
We consider a network flow problem defined by a directed graph \( G = (V, E) \), where \( V \) is the set of nodes and \( E \) is the set of arcs.
For each $e \in E$, there are specified costs \( c_e \in \mathbb{R}_+ \) per unit of flow and a flow capacity \( u_e \in \mathbb{R}_+ \). Denoting by $x_e$ the variables representing the flow along an arc $e \in E$, we have the capacity constraints
\begin{equation}
\label{constraint_1}
 0 \leq x_e \leq u_e, \quad e \in E.
\end{equation}
The \emph{net flow} at some node $v \in V$ is defined as
\[
\netflow(x;v) \coloneqq \sum_{\substack{u \in V \\ uv \in E}} x_{uv} - \sum_{\substack{u \in V \\ vu \in E}} x_{vu}.
\]
We assume that the set \( V \) of nodes is partitioned into the set \( V^\text{supply} \) of supply nodes, where the net flow is expected to be non-positive, and the set \( V^\text{demand} \) of demand nodes, where the net flow must be positive.
For simplicity, we do not explicitly consider nodes with zero net flow.
For each demand node \( v \in V^\text{demand} \), we assume that there is a fixed positive demand \( b_v > 0 \).
This can be expressed by the linear constraints
\begin{equation}
\label{constraint_2}
\netflow(x;v) = b_v, \quad v \in V^\text{demand}.
\end{equation}
Another task in the optimization process is determining the supply capacities.
To achieve this, we introduce variables \( z_v \) for \( v \in V^\text{supply} \) and add the constraints
\begin{equation}
\label{constraint_3}
0 \leq -\netflow(x;v) \leq z_v, \quad v \in V^\text{supply}.
\end{equation}
These constraints are linear in the variables \( x_e \) and \( z_v \). 
To establish one unit of supply capacity at a node \( v \in V^\text{supply} \) incurs a cost of \( a_v \).
There are two types of costs: the operating costs
\[
c(x) \coloneqq \sum_{e \in E} c_e x_e
\]
for running the network, and the establishment costs
\[
a(z) \coloneqq \sum_{v \in V^\text{supply}} a_v z_v
\]
for setting up supply capacities.

Our first objective is to minimize the total costs
\[
f_1(z, x) \coloneqq c(x) + a(z).
\]
So far, we have a slightly generalized min-cost-flow problem.
Another objective relates to the stability of the network, interpreted as the aim to avoid fully utilizing the capacities of the arcs. If a specific percentage \( \tau \in (0,1) \) of the respective capacity is exceeded, this excess will be accumulated as follows, where \( w^+ \coloneqq \max\{0, w\} \):
\[
s(x) \coloneqq \sum_{e \in E} (x_e - \tau u_e)^+.
\]
Additionally, we seek to avoid fully utilizing the supply capacities. 
If the supply, defined as the negative of the net flow, exceeds a specific percentage \( \mu \in (0,1) \) of the respective capacity \( z_v \), the excess will be accumulated as follows:
\[
t(z, x) \coloneqq \sum_{v \in V^\text{supply}} (-\netflow(x;v) - \mu z_v)^+.
\]
The second objective function, which quantifies the degree of instability in the network, is a weighted sum of both cumulative excesses:
\[
f_2(z, x) \coloneqq \gamma_1 s(x) + \gamma_2 t(z, x),
\]
where \( \gamma_1, \gamma_2 \geq 0 \) are the respective weights.
Note that the minimization of a term \( w^+ \) can be replaced by minimizing an auxiliary variable \( v \) subject to the constraints \( v \geq 0 \) and \( v \geq w \). Therefore, the following can be seen as a multi-objective linear program, where we set \( f = (f_1, f_2)^T \):
\leqnomode	
	\begin{gather}\label{eq:nscd1}\tag{NSCD}
		\min_{z, x} f(z, x) \quad \text{subject to} \quad \eqref{constraint_1}, \eqref{constraint_2}, \eqref{constraint_3}.
	\end{gather}
\reqnomode
This problem is inherently multi-stage in nature, involving two distinct stages. In typical applications, supply capacities are fixed initially, followed by decisions on how to operate the network later. Alternatively, the problem can be viewed as requiring a one-time fixation of supply capacities while accommodating multiple operational scenarios. This creates a \emph{preference for flexibility} \cite{Kreps79}, suggesting that the choice of supply capacities should maximize the flexibility in operating the network. To address this, the first-stage decision should be based on a corresponding set linear program
\leqnomode	
	\begin{gather}\label{eq:nscd2}\tag{NSCD'}
		\opt_z F(z),
	\end{gather}
\reqnomode
where the set-valued objective map $F$ is defined by
\[
F(z) \coloneqq \{ y \in \mathbb{R}^2 : \exists x \in \mathbb{R}^{|E|} \text{ such that } y \geq f(z,x), \, \eqref{constraint_1}, \, \eqref{constraint_2}, \, \eqref{constraint_3} \}.
\]
The set linear program \eqref{eq:nscd2} can be seen as a multi-stage variant of the multi-objective linear program \eqref{eq:nscd1}.

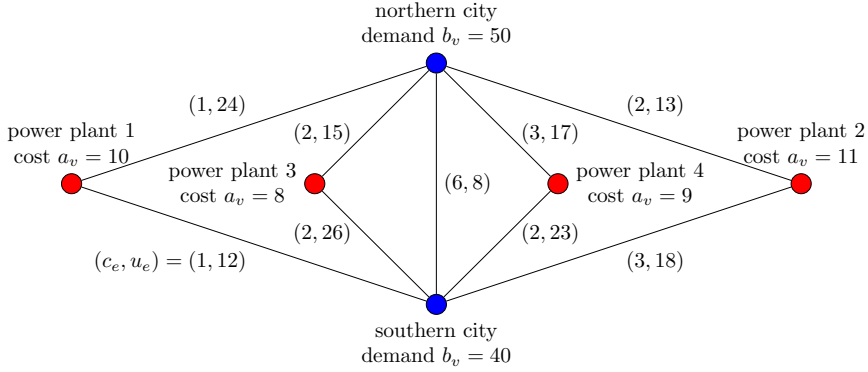
\begin{figure}[hpt]
	\begin{center}
	\scalebox{.8}{ 
	\begin{tikzpicture}
	  \node[draw, circle, fill=red, minimum size=3mm, 
	  label={[align=center]above:{power plant 1 \\ cost $a_v=10$}}] (n1) at (180:6) {};
	  \node[draw, circle, fill=red, minimum size=2mm, 
	  label={[align=center]above:{power plant 2 \\ cost $a_v=11$}}] (n2) at (0:6) {};
	  \node[draw, circle, fill=red, minimum size=2mm, 
	  label={[align=center]left:{power plant 3 \\ cost $a_v=8$}}] (n3) at (180:2) {};
	  \node[draw, circle, fill=red, minimum size=2mm, 
	  label={[align=center]right:{power plant 4 \\ cost $a_v=9$}}] (n4) at (0:2) {};
	  \node[draw, circle, fill=blue, minimum size=2mm, 
	  label={[align=center]above:{northern city \\ demand $b_v = 50$}}] (nn) at (90:2) {};
	  \node[draw, circle, fill=blue, minimum size=2mm, 
	  label={[align=center]below:{southern city \\ demand $b_v = 40$}}] (ns) at (270:2) {};
           
	  \draw (n1) -- node[pos=0.5, above left] {$(1,24)$} (nn);
	  \draw (n1) -- node[pos=0.5, below left] {$(c_e,u_e)=(1,12)$} (ns);
	  \draw (n2) -- node[pos=0.5, above right] {$(2,13)$} (nn);
	  \draw (n2) -- node[pos=0.5, below right] {$(3,18)$} (ns);
	  \draw (n3) -- node[pos=0.4, left] {$(2,15)\;$} (nn);
	  \draw (n3) -- node[pos=0.4, left] {$(2,26)\;$} (ns);
	  \draw (n4) -- node[pos=0.4, right] {$\;(3,17)$} (nn);
	  \draw (n4) -- node[pos=0.4, right] {$\;(2,23)$} (ns);
	  \draw (nn) -- node[pos=0.5, right] {$(6,8)$} (ns);
	\end{tikzpicture}
	}
	\end{center}
	\caption{The digraph and the data of Example \ref{ex:energy}.}
	\label{fig:energy1}
\end{figure}
	
\begin{ex} \label{ex:energy} 	
Consider the electricity network depicted in Figure \ref{fig:energy1}. The $k=2$ demand nodes represent two districts of a city. The capacities of the $l=4$ power plants, which are the supply nodes, need to be determined. The digraph $G$ has $n=6$ nodes and $m=9$ arcs. The costs $c_e$ of transporting one unit of energy along an interconnection wire $e \in E$, the capacities $u_e$ of the interconnection wires, the energy demand $b_v$ in the two districts, and the costs $a_v$ for establishing one unit of power plant capacity at the four supply nodes are given in Figure \ref{fig:energy1}. 
The decision process and its result for the parameters $\tau = 0.8$, $\mu = 0.9$, $\gamma_1 = 1$, and $\gamma_2 = 3$ are depicted in Figure \ref{fig:energy2}, which can be explained as follows: 

The current finite set \( Y \)---the points selected by the decision-maker---is displayed by red circles.
 An orange set represents the value $v_F(Y) = F(z)$ for some optimizer $z$. The set $v_F(Y)$ is displayed in yellow color if there is no $z$ such that $F(z) \supseteq v_F(Y)$. In this case, the set \( v_F(Y) \) represents the options available to the decision-maker for selecting the next point, which is intended to be contained within the value $F(z)$ of an optimizer $z$ that will be constructed.
The eight pictures in Figure \ref{fig:energy2}, numbered row-wise from left to right, can be described as follows:

(1) Initially, we have $Y = \emptyset$. Therefore, all points of the optimal value can be chosen. The yellow set $v_F(\emptyset)$ of options coincides with the optimal value $\optval_{z} F(z)$, which is displayed in gray in the background.

(2) The decision maker chooses the point $y^1 \coloneqq (1071, 0)^T$, which is the vertex of the optimal value with the best network stability. This leads to the optimizer $z^1 \coloneqq (22.2, 5.8, 36.4, 35.6)^T$. The components of $z^1$ represent the supply capacities of the four power plants. The orange set $v_F(\Set{y^1}) = F(z^1)$ shows the options for second-stage decisions, which relate to how the network can be operated. The total supply capacity is $100.0$.

(3) The decision-maker deletes the previously selected point $y^1$ because it leads to relatively high costs for all options of operating the network.
Once again, all points from the optimal value are available for selection.

(4) The decision-maker selects the point $y^2 \coloneqq (1075, 3.5)$, which is in the interior of the optimal value set. This selection does not yet result in an optimizer. Another point must be chosen from the yellow set $v_F(\Set{y^2})$, which displays the options for a second selection.

(5) The decision maker selects $y^3 \coloneqq (1045,11.3)$ as a second point. This selection results in the optimizer $z^2 \coloneqq (29.6, 0.0, 36.4, 32.7)$, which has a total supply capacity of $98.7$. The set $v_F(\Set{y^2,y^3})=F(z^2)$ is displayed in orange.

(6) The decision-maker is still not satisfied. The previously selected point $y^2$ is deleted, resulting in new options $v_F(\Set{y^3})$ being displayed.

(7) The new second point $y^4 \coloneqq (1035, 14.4)$ is chosen by the decision-maker, but this selection does not yet lead to an optimizer.

(8) The decision-maker selects $y^5 \coloneqq (1064, 6.1)$ as a third point. This results in the optimizer $z^3 \coloneqq (28.4, 0.0, 36.4, 32.8)$, with a total supply capacity of $97.6$. The decision-maker is satisfied with this optimizer, and the process is complete.

One can observe that the second and third optimizers identified in the process provide more flexibility in operating the network than the first optimizer. This is why the decision-maker forgoes Pareto-optimal options.
\end{ex}

\def\ww{.52} 
\def\www{1.04} 
\begin{figure}
    \begin{center}
        \begin{minipage}{\www\textwidth}
            \includegraphics[width=\ww\textwidth]{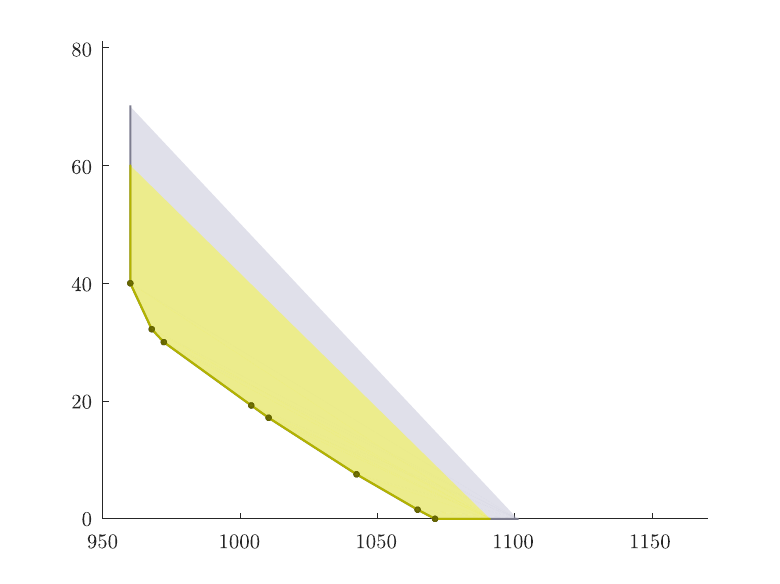}
            \includegraphics[width=\ww\textwidth]{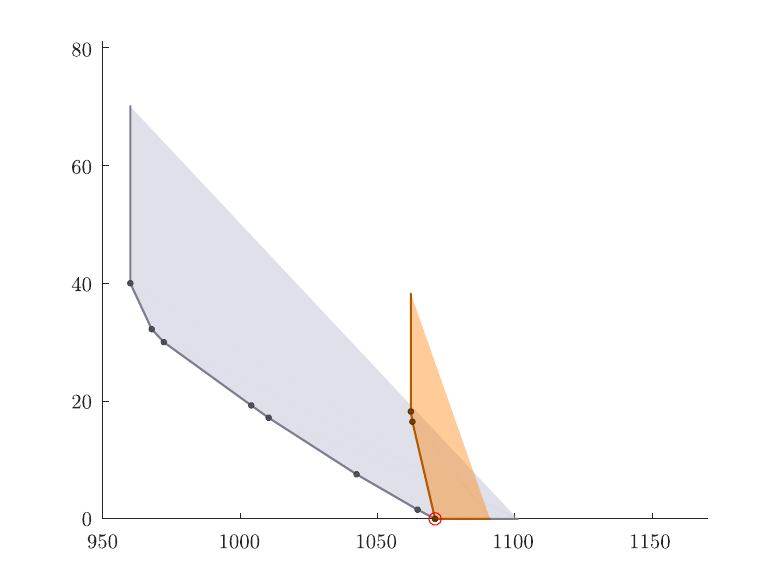}
        \end{minipage}%
        \hfill 
        \begin{minipage}{\www\textwidth}
            \includegraphics[width=\ww\textwidth]{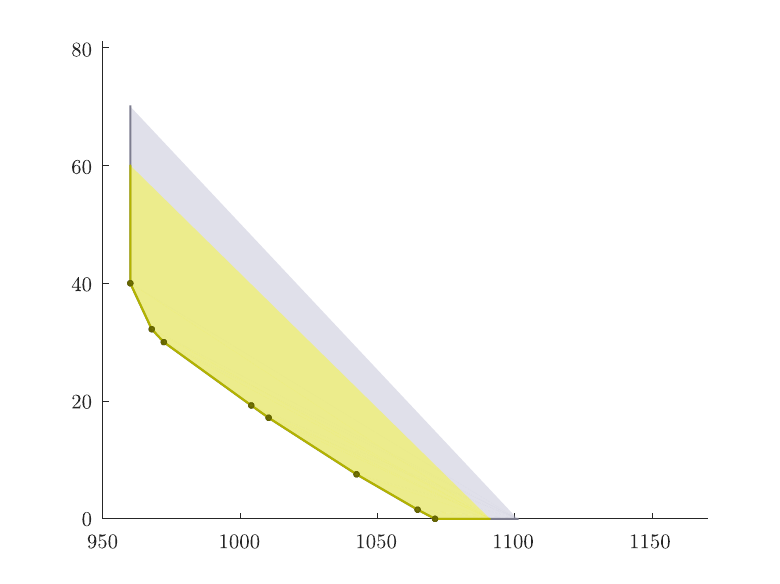}
            \includegraphics[width=\ww\textwidth]{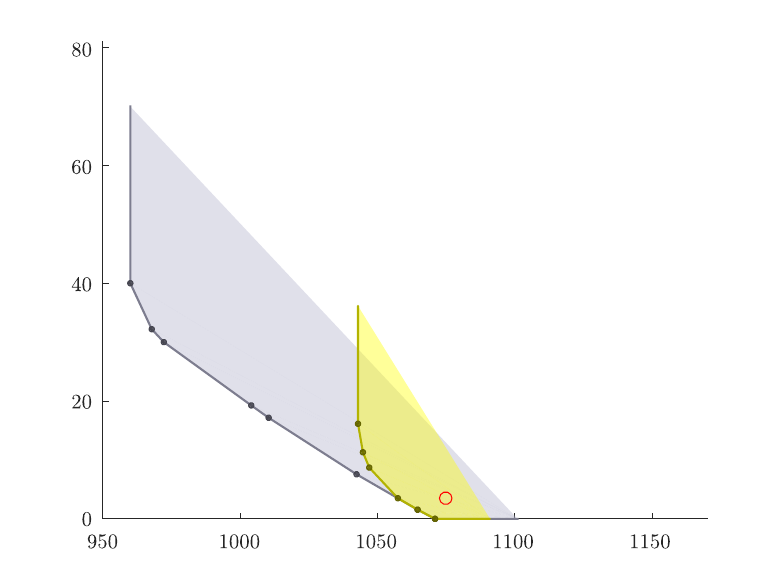}
        \end{minipage}%
        \hfill
        \begin{minipage}{\www\textwidth}
            \includegraphics[width=\ww\textwidth]{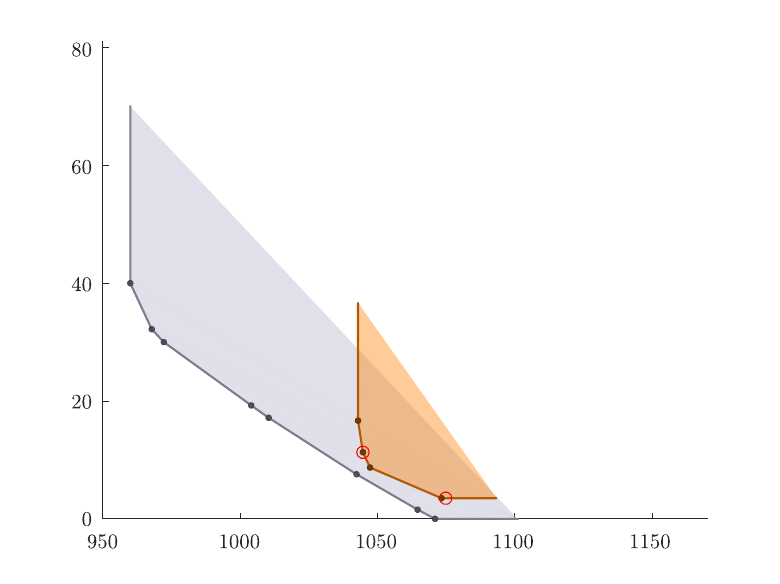}
            \includegraphics[width=\ww\textwidth]{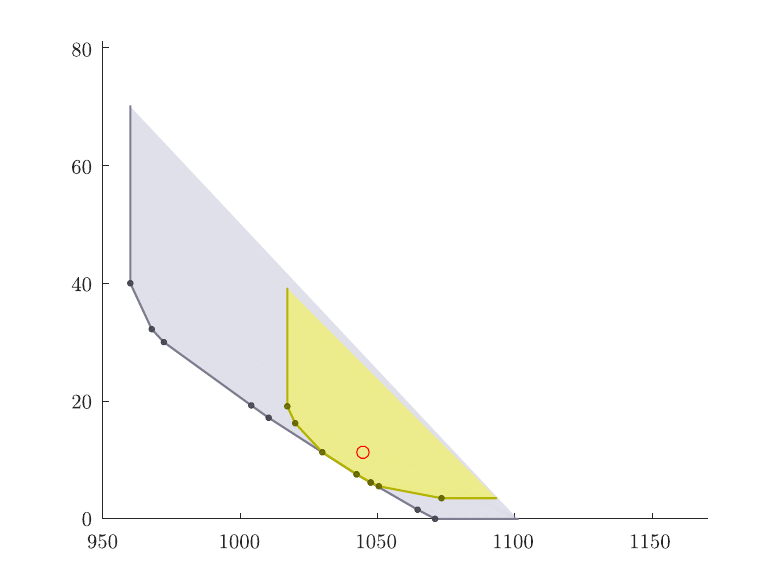}
        \end{minipage}%
        \hfill
        \begin{minipage}{\www\textwidth}
            \includegraphics[width=\ww\textwidth]{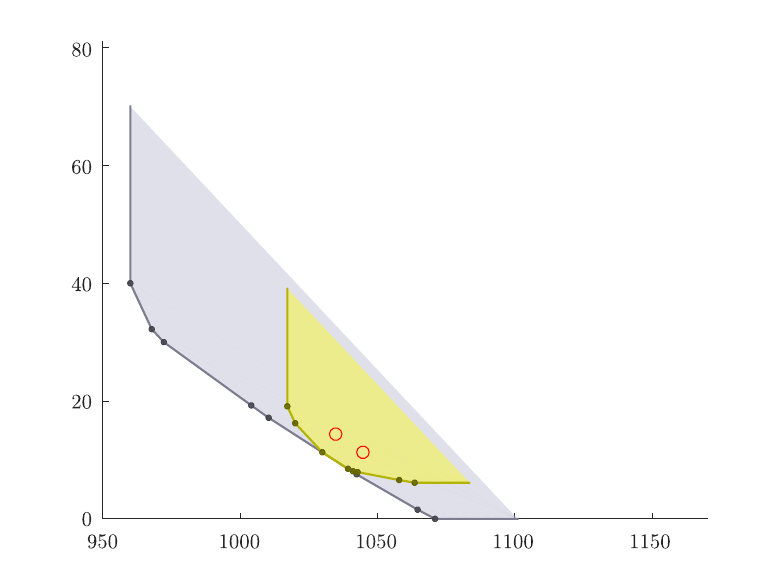}
            \includegraphics[width=\ww\textwidth]{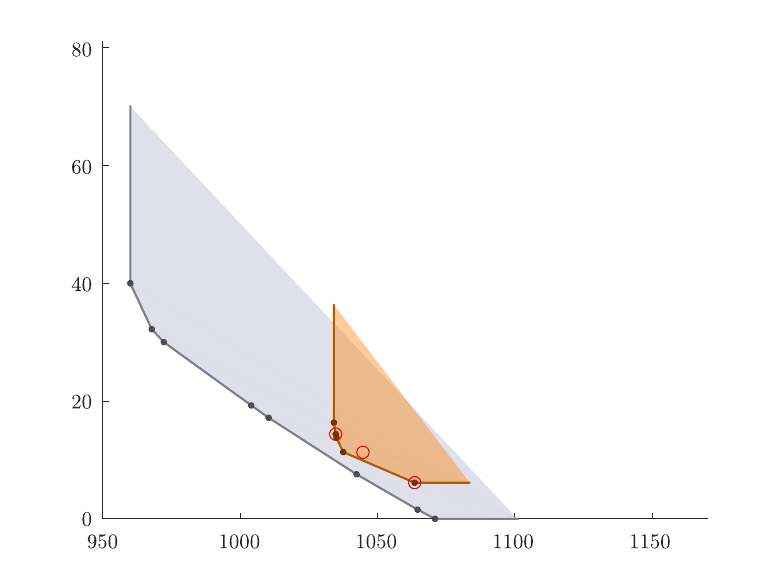}
        \end{minipage}
    \end{center}
    \caption{Decision process from Example \ref{ex:energy}. }\label{fig:energy2}
\end{figure}

\section{Implementation}\label{sec:impl}

In this section, we discuss the two primary operations of the proposed method: the computation of $v_F(Y)$ for finite sets $Y$, and the stopping condition of the loop, i.e., computing $x$ such that $F(x) \supseteq v_F(Y)$ or establishing that no such $x$ exists. For this, we assume that $F$ is polyhedral convex with a nonempty domain, and that optimizers exist, i.e., $K = G(0)$.

Polyhedral calculus, as discussed in \cite{CirLoeWei19}, asserts that the result of operations such as the Cartesian product of finitely many convex polyhedra, the intersection of finitely many convex polyhedra, and an orthogonal projection of a convex polyhedron can be easily represented using P-representations, provided the input polyhedra are given in P-representation form. In order to compute a V-representation and an H-representation of a P-represented convex polyhedron, a polyhedral projection problem, or equivalently, a multiple objective linear program must be solved \cite{LoeWei16}. Although polyhedral calculus remains tractable for high-dimensional convex polyhedra, the computational effort required for polyhedral projection typically increases exponentially with the polyhedron's dimension. Both polyhedral calculus and polyhedral projection can be realized using the software \texttt{bensolve tools} \cite{bt}, which was employed for the computations in Example \ref{ex:energy}.

If as usual $F$ is given by its graph, for some convex polyhedron $X \subseteq \R^n$, we have

$$ \optval_{x \in X} F(x) \coloneqq \bigcup_{x \in X} F(x) = \Pi_{-q} (\gr F \cap (X \times \R^q)),$$
where, for a set $S \subseteq \R^n \times \R^q$, $\Pi_{-q}(S)$ denotes the orthogonal projection to the last $q$ coordinates of $\R^n \times \R^q$.
Likewise, for some $y \in \R^q$ we have
$$ F^{-1}(y) \coloneqq \Set{x \in \R^n \given y \in F(x)} = \Pi_n(\gr F \cap (\R^n \times \Set{y})),$$ 
and $\Pi_{n}(S)$ denotes the orthogonal projection to the first $n$ coordinates. 
Using this, the values of $v_F$ at a finite set $Y=\Set{y^1,\dots,y^k}$ can be composed as
$$ v_F(Y) = F\left(\bigcap_{i=1}^k F^{-1}(y^i)\right).$$
Thus, a P-representation of $v_F(Y)$ can be easily obtained using polyhedral calculus, while both a V-representation and an H-representation of $v_F(Y)$ can be derived through polyhedral projection in $\mathbb{R}^q$, where $q$ denotes the number of objectives in our problem. In many applications, this dimension $q$ is low, typically ranging from $q=2$ to $q=4$.

Let \( v_F(Y) = \{v^1,\dots,v^\ell\} + K \) be a V-representation of \( v_F(Y) \), where \( K \) denotes the natural ordering cone introduced in Section \ref{sec:exist}, which was shown to be the recession cone of $v_F(Y)$ in Proposition \ref{prop:polyh}. By polyhedral calculus, we obtain a P-representation of the convex polyhedron 
\[
M \coloneqq \bigcap_{i=1}^\ell F^{-1}(v^i).
\]
Note that \( x \in M \) implies \( \{v^1,\dots,v^\ell\} \subseteq F(x) \). Using the assumption \( K = G(0) \), from \( x \in M \) we conclude 
\[
v_F(Y) = \{v^1,\dots,v^\ell\} + K \subseteq F(x) + K = F(x) + G(0) = F(x).
\]
The equivalence 
\[
x \in M \quad \iff \quad F(x) \supseteq v_F(Y)
\]
is now straightforward to see. This means that the optimality test can be performed by computing a V-representation of \( v_F(Y) \), obtaining a P-representation of \( M \) via polyhedral calculus, and solving a linear program to determine some \( x \in M \) or to establish that \( M \) is empty.


\end{document}